\documentclass{amsart}
\usepackage[utf8]{inputenc}
\usepackage{amsmath,amsthm}
\usepackage{comment}
\usepackage{amssymb}
\usepackage{stmaryrd}
\usepackage{hyperref}
\usepackage{color}
\usepackage[left=3cm,right=3cm,top=3cm,bottom=4.5cm]{geometry}
\usepackage{pdfsync}
\usepackage{tikz-cd}
\usepackage{url}
\usepackage{enumitem}
\hypersetup{linktoc=all}
\usepackage[capitalise,noabbrev]{cleveref}

\newcommand{\defterm}[1]{\emph{#1}}

\numberwithin{equation}{section}

\theoremstyle{plain}
	\newtheorem{thm}[equation]{Theorem}
	\newtheorem*{thm*}{Theorem}
	\newtheorem{cor}[equation]{Corollary}
	\newtheorem*{cor*}{Corollary}
	
	\newtheorem*{prop*}{Proposition}
	\newtheorem{lem}[equation]{Lemma}
	\newtheorem*{lem*}{Lemma}
	
	\newtheorem*{ex*}{Exercise}
	
	\newtheorem*{claim*}{Claim}
	
	\newtheorem*{question*}{Question}
	
	\newtheorem*{fact*}{Fact}
\theoremstyle{definition}
	\newtheorem{Def}[equation]{Definition}
	\newtheorem*{Def*}{Definition}
	\newtheorem{obs}[equation]{Observation}
	\newtheorem*{obs*}{Observation}
	\newtheorem{rmk}[equation]{Remark}
	\newtheorem*{rmk*}{Remark}
	
	\newtheorem{soln*}{Solution}
	
	\newtheorem*{note*}{Note}
	\newtheorem{eg}[equation]{Example}
	\newtheorem*{eg*}{Example}	
	
	\newtheorem*{construction*}{Construction}
	\newtheorem{warning}[equation]{Warning}
	\newtheorem*{warning*}{Warning}
	
	\newtheorem*{conj*}{Conjecture}

\newcommand{\nats}{\mathbb{N}}

\newcommand{\id}{\mathrm{id}}

\newcommand{\Hom}{\mathrm{Hom}}

\newcommand{\Ob}{\operatorname{Ob}}

\newcommand{\const}{\mathrm{const}}

\newcommand{\Psh}{\mathsf{Psh}}

\newcommand{\Cat}{\mathsf{Cat}}

\newcommand{\calC}{\mathcal{C}}
\newcommand{\calD}{\mathcal{D}}

\newcommand{\Spaces}{\mathsf{Spaces}}

\newcommand{\Gaunt}{\mathsf{Gaunt}}
\newcommand{\nat}{\mathrm{nat}}

\newcommand{\sCat}{\mathsf{sCat}}
\newcommand{\lax}{\mathrm{lax}}
\newcommand{\inc}{\mathrm{inc}}

\title{Cubes are dense in $(\infty,\infty)$-categories}
\author{Tim Campion}
\date{September 2022}

\begin{document}

\begin{abstract}
    We show that the strict 1-category $\square$ of cubes -- defined to be the full subcategory of strict $\omega$-categories whose objects are the Gray tensor powers of the arrow category -- are dense in the $(\infty,1)$-category $\mathsf{Cat}_\omega$ of weak $(\infty,\infty)$-categories, in both Rezk-complete and incomplete variants. More precisely, we show that Joyal's category $\Theta$ is contained in the idempotent completion of $\square$, and in fact that the idempotent completion of $\square$ is closed under suspensions and wedge sums. This result extends a theorem of Campbell and Maehara in dimension 2. Following Campbell and Maehara's program, we will in future work apply this result to give a new construction of the Gray tensor product of weak $(\infty,n)$-categories.
\end{abstract}

\maketitle

\tableofcontents

\section{Introduction}
For each $N \in \nats \cup \{\omega\}$, Joyal's category $\Theta_N$ (we write $\Theta$ for $\Theta_\omega$) is a certain full subcategory of the strict 1-category $\Gaunt_N$ of gaunt $N$-categories (cf. \cite{berger}). As such, $\Theta_N$ is a full subcategory of both the 1-category $\sCat_N$ of strict $N$-categories, and of the $(\infty,1)$-category $\Cat_N$ of $(\infty,N)$-categories. Moreover, both inclusion functors $\Theta_N \to \sCat_N$ and $\Theta_N \to \Cat_N$ are \defterm{dense} in the sense that the induced nerve functors $\sCat_N \to \Psh(\Theta_N)$, $\Cat_N \to \Psh(\Theta_N)$ are fully faithful. Thus it is possible to describe each of the $(\infty,1)$-categories $\sCat_N$ and $\Cat_N$ as localizations of the presheaf category $\Psh(\Theta_N)$. In the weak case such a description was first given by Rezk \cite{rezk} (in the strict case, the description is folklore).

There are many possible alternative choices of a full, dense subcategory $A \subseteq \Cat_N$ of potential use for modeling $(\infty,N)$-categories, several of which are studied in \cite{bsp}. Among the possible options, $\Theta_N$ stands out as the ``minimal" choice.\footnote{Indeed, if $A$ is any idempotent-complete, dense full subcategory of $\Cat_N$, then $\Theta_N \subseteq A$ \cite{MO-theta}.} But for some purposes, it is preferable to work with presheaves over a larger site $A$.

For example, when studying the Gray tensor product, it is convenient to work with some dense full subcategory $A \subseteq \Cat_\omega$ which (unlike $\Theta$) is closed under the Gray tensor product. The (lax) Gray tensor product is a biclosed monoidal structure $\otimes^\lax$ on $\sCat_N$ \cite{abs}, whose internal hom classifies $N$-functors and lax natural transformations between them. 
The minimal possible choice would be the category $\square = \{\square^n := [1]^{\otimes^\lax n} \mid n \in \nats\}$ of \defterm{$\omega$-categorical cubes}, 
comprising the the Gray tensor powers $\square^n = [1]^{\otimes^\lax n}$ of the arrow category $\square^1 = [1]$. The first few cubes are pictured below:

\begin{align*}
    \square^0 &= 
    \begin{tikzcd}[ampersand replacement = \&]
        \bullet
    \end{tikzcd}
    \\
    \square^1 &=
    \begin{tikzcd}[ampersand replacement = \&]
        \bullet \ar[r] \& \bullet
    \end{tikzcd}
    \\
    \square^2 &=
    \begin{tikzcd}[ampersand replacement = \&]
        \bullet \ar[r] \& \bullet \\
        \bullet \ar[r] \ar[u,""{name = B, below}] \& \bullet \ar[u, ""{name = T, below}]
        \ar[Rightarrow,from = B, to = T]
    \end{tikzcd}
\end{align*}

Following a program of Campbell and Maehara in dimension 2 \cite{campbell-maehara}, the goal of this note is to prove that the full subcategory inclusion $\square \subset \Cat_\omega$ is in fact dense (\cref{cor:dense} below). In later work \cite{campion-model}, we will describe the resulting localization exhibiting $\Cat_\omega$ as a localization of presheaves on a site slightly larger than $\square$. Following \cite{campbell-maehara}, we will then use this presentation to study the Gray tensor product in the non-strict setting in \cite{campion-gray}.

Following \cite{campbell-maehara}, we reduce the density question to a question in strict $\omega$-category theory as follows. Note that a full subcategory inclusion $\calC \subseteq \calD$ is dense if and only if its idempotent completion $\widetilde \calC \subseteq \calD$ is dense. So to achieve our aim, it will suffice to show that the idempotent completion $\widetilde \square \subseteq \Cat_\omega$ is dense (\cref{cor:dense} below). Moreover if $\calC \subseteq \calD$ is a dense full subcategory and $\calC \subseteq \calC' \subseteq \calD$ is a larger full subcategory, then $\calC' \subseteq \calD$ is dense. So, because $\Theta_N \subseteq \Cat_N$ is dense (\cref{lem:theta-dense} -- this is well known for $N < \omega$), it will suffice to show that $\Theta \subseteq \widetilde \square$ (\cref{cor:dense} below). This is a question which takes place entirely in the setting of the strict 1-category $\Gaunt_\omega \subset \sCat_\omega$ of gaunt $\omega$-categories; no weak $\infty$-category theory is required to address this question. For this reason, there is no further weak $\infty$-category theory in this note. Moreover, in this setting, known facts about the Gray tensor product of strict $\omega$-categories (such as the fact that it is a monoidal biclosed structure) may be brought to bear.

Finally, to prove that $\Theta \subseteq \widetilde \square$, we depart from the approach of \cite{campbell-maehara}.\footnote{Although Campbell and Maehara's proof that $\Theta_2 \subseteq \widetilde \square$ has not appeared in full detail, it has been described as a technical, combinatorics-heavy proof.} In our approach, we use a characterization of $\Theta$ (\cref{lem:theta-ind} below) which roughly says that $\Theta \subseteq \sCat_\omega$ is the smallest full subcategory which contains $\square^0$, is closed under suspension, and is closed under wedge sums. Thus it roughly suffices to show that $\widetilde \square \subset \sCat_\omega$ is closed under suspension and wedge sums. We say ``roughly" because in order to define the wedge sum, a bit of additional structure in the form of a bipointing is required; this is discussed in \cref{sec:theta}. Our main theorems are stated in \cref{sec:cubes}. Among these, those corollaries which pertain to weak $(\infty,N)$-categories are quarantined in \cref{subsec:weak}. The proofs that $\widetilde \square$ is closed under wedge sums (\cref{thm:wedge}) and suspensions (\cref{thm:susp}) occupy \cref{sec:wedge} and \cref{sec:susp} respectively. \cref{sec:app} proves a lemma about cube categories which we could not find in the literature.

\subsection{Notation}\label{subsec:notation}
For $N \in \nats \cup \{\omega\}$, we let $\sCat_N$ denote the strict 1-category of strict $N$-categories. We let $\Theta$ denote Joyal's category $\Theta$, as defined in e.g. \cite{berger}, and let $\Theta_N = \Theta \cap \sCat_N$. We denote by $\otimes^\lax : \sCat_M \times \sCat_N \to \sCat_{M+N}$ the lax Gray tensor product (see \cite{abs} or \cite{ara-malt}). We let $\square^1 \in \sCat_1$ denote the arrow category, and for $n \in \nats$, we let $\square^n = (\square^1)^{\otimes^\lax n}\in \sCat_n$ denote its n-fold power under the Gray tensor product -- the \defterm{$n$-cube}. We write $\partial \square^n$ for the $(n-1)$-skeleton of $\square^n$.

We write $\Spaces$ for the $(\infty,1)$-category of spaces. We use $\Psh(\calC)$ to denote $\Spaces$-valued presheaves on an $(\infty,1)$-category $\calC$. For $N \in \nats \cup \{\omega\}$, we let $\Cat_N^\inc$ denote the localization of $\Psh(\Theta_N)$ at the spine inclusions. 
Then $\Cat_N$ is the localization of $\Cat_N^\inc$ at the Rezk maps \cite{rezk}.
We note in \cref{rmk:defs} below that we have $\Cat_\omega^\inc = \varprojlim_N \Cat_N^\inc$ and $\Cat_\omega = \varprojlim_N \Cat_N$, where the limit is taken in $\Cat$ itself along the forgetful functors $\Cat_{N+1}^{(\inc)} \to \Cat_N^{(\inc)}$, which throw away non-invertible $N+1$-cells. So $\Cat_\omega$ is the $(\infty,1)$-category of $(\infty,\infty)$-categories with \emph{inductive} equivalences, rather than \emph{coinductive} equivalences.

We say that a functor of $(\infty,1)$-categories $F : \calC \to \calD$ is \defterm{dense} if the induced functor $\calD \to \Psh(\calC)$ is fully faithful. This notion has been called \emph{strongly generating} in \cite{htt}, but we have followed recent revisions of \cite{bsp} in reverting to the term \emph{dense}, which mirrors the established terminology in the strict 1-categorical literature.

We denote by $\widetilde{\calC}$ the idempotent-completion of an $(\infty,1)$-category $\calC$. Note that the idempotent completion of a strict 1-category is again a strict 1-category.

We write $\Pr^L$ for the $(\infty,1)$-category of presentable $\infty$-categories in the sense of \cite{htt}, with morphisms the left adjoint functors.

\subsection{Acknowledgements}
I would like to thank Alexander Campbell, Sina Hazratpour, Yuki Maehara, Emily Riehl, Maru Sarazola, and Jonathan Weinberger for helpful conversations. I'm grateful for the support of the ARO under MURI Grant W911NF-20-1-0082.

\section{The inductive structure of $\Theta_N$}\label{sec:theta}
The main goal of this section is to introduce some terminology allowing us to state \cref{lem:theta-ind}, which gives an inductive approach to understanding objects of $\Theta_N$ in terms of two building blocks : wedge sums (\cref{def:wedge}) and suspensions (\cref{def:susp}). The main annoyance is that the wedge sum is only functorial with respect to \emph{bipointed maps}, which necessitates waffling a bit about bipointed objects (\cref{def:bipointed}). Importantly, both wedge sum and suspension continue to make sense in $\sCat_N$ more generally, and in particular we may apply these constructions to cubes. The wedge sum is manifestly defined by a pushout formula which we will use. The suspension may also be defined by a pushout formula (\cref{lem:susp-po}), which we will also use.

\begin{Def}\label{def:bipointed}
Let $\calC$ be a category with a terminal object denoted $\square^0$. Then $\calC_{\ast,\ast}$, the category of \emph{bipointed objects of $\calC$}, is the category whose objects comprise an object $C \in \calC$, along with two morphisms $\square^0 \rightrightarrows \calC$. The morphsims of $\calC_{\ast,\ast}$ preserve the bipointing. 
\end{Def}

\begin{eg}
Let $N \in \nats \cup \{\omega\}$ and $\theta \in \Theta_N$. Then the set of objects of $\theta$ is canonically identified with $\{0,1,\dots,k\}$ for some $k \in \nats$. The \defterm{natural bipointing} on $\theta$ is defined to be the object $\square^0 \rightrightarrows \theta$ of $\Theta_{\ast,\ast}$ which picks out $0$ and $k$. We denote by $\Theta_{N,\ast,\ast}^\nat$ the full subcategory of $\Theta_{N,\ast,\ast}$ on the naturally-bipointed objects. 
\end{eg}

\begin{eg}
Let $\square^n$ be a cube. Then the set of objects of $\square^n$ is canonically identified with the powerset lattice $\{0,1\}^n$. The \defterm{natural bipointing} of $\square^n$ is defined to be the object $\square^0 \rightrightarrows \square^n$ of $\Gaunt_{N, \ast,\ast}$ which picks out the bottom object $\vec 0 = (0,\dots,0)$ and the top object $\vec 1 = (1,\dots,1)$. We denote by $\square^\nat_{\ast,\ast}$ the full subcategory of $\Gaunt_{\ast,\ast}$ on the naturally-bipointed cubes.  
\end{eg}

\begin{rmk}\label{rmk:tensor-bipoint}
Let $N \in \nats \cup \{\omega\}$. The Gray tensor product lifts naturally to a bifunctor $\otimes^\lax : \sCat_{N,\ast,\ast} \times \sCat_{N,\ast,\ast} \to \sCat_{N,\ast,\ast}$, where the bipointing of $A \otimes^\lax B$ is the ``diagonal" bipointing. That is, the bipointings $\partial \square^1 \to A$, $\partial \square^1 \to B$ induce a map $(\partial \square^1) \otimes^\lax (\partial \square^1) \to A \otimes^\lax B$. Precomposing with the diagonal map $\partial \square^1 \to (\partial \square^1) \otimes^\lax (\partial \square^1)$, $(0 \mapsto (0,0),\, 1 \mapsto (1,1))$, we obtain a bipointing of $A \otimes^\lax B$, which is functorial in bipointed maps.
\end{rmk}

\begin{Def}\label{def:wedge}
Let $\calC$ be a category, and let $C, D \in \calC_{\ast,\ast}$ be bipointed objects of $\calC$. The \defterm{wedge sum} of $C$ and $D$, denoted $C \vee D$, is defined (when it exists) to be the pushout $C \vee D = C \cup_{\square^0} D$, where the pushout diagram uses the second point of $C$ and the first point of $D$. This pushout is taken in $\calC$ (and again, may fail to exist in general). The wedge sum $C \vee D \in \calC$ is lifted to $\calC_{\ast,\ast}$ by equipping it with the bipointing whose first point is the first point of $C$, and whose second point is the second point of $D$.

Observe that, insofar as the wedge sum exists, it defines a bifunctor $\vee : \calC_{\ast,\ast} \times \calC_{\ast,\ast} \to \calC_{\ast,\ast}$.
\end{Def}

\begin{eg}\label{eg:wedge-theta}
Let $N \in \nats \cup \{\omega\}$. Wedge sums exist in $\sCat_N$, defining a bifunctor $\vee : \sCat_{N,\ast,\ast} \times \sCat_{N,\ast,\ast} \to \sCat_{N,\ast,\ast}$. This restricts to a bifunctor $\vee : \Theta_{N,\ast,\ast}^\nat \times \Theta_{N,\ast,\ast}^\nat \to \Theta_{N,\ast,\ast}^\nat$.
\end{eg}

\begin{Def}\label{def:susp}
Let $N \in \nats \cup \{\omega\}$, and let $C \in \sCat_N$. The \defterm{suspension} of $C$, denoted $\Sigma C \in \sCat_{N+1}$, has two objects $0,1$, with $\Hom_{\Sigma C}(0,0) = \Hom_{\Sigma C}(1,1) = \square^0$, $\Hom_{\Sigma C}(1,0) = \emptyset$, and $\Hom_{\Sigma C}(0,1) = C$.

$\Sigma C$ is regarded canonically as a bipointed object by taking the two points to be $0$ and $1$. In this way, $\Sigma$ yields a functor $\sCat_N \to \sCat_{N+1,\ast,\ast}$, which restricts to a functor $\Gaunt_N \to \Gaunt_{N+1,\ast,\ast}$ and further restricts to $\Theta_N \to \Theta_{N+1,\ast,\ast}^\nat$.
\end{Def}

\begin{lem}\label{lem:susp-univ}
The suspension $\Sigma C$ has the following universal property. The data of a functor $\Sigma C \to D$ is equivalent to the choice of two objects $d_0,d_1 \in D$ (the images of $0,1$ respectively) and the choice of a functor $C \to D(d_0,d_1)$ (specifying the action of the functor on the nontrivial hom-category). \qed
\end{lem}

\begin{lem}\label{lem:susp-po}
Let $N \in \nats \cup \{\omega\}$ and $C \in \sCat_N$. Then there is a canonical identification 
\[\Sigma C = (\square^1 \otimes^\lax C) \cup_{\partial \square^1 \otimes^\lax C} \partial \square^1\]
\end{lem}
\begin{proof}
This is \cite[Cor. B.6.6]{ara-malt}.
\end{proof}

\begin{lem}\label{lem:theta-ind}
Let $N \in \nats \cup \{\omega\}$. Then $\Theta_{N,\ast,\ast}^\nat$ is the smallest full subcategory $\calC \subseteq \sCat_{N,\ast,\ast}$ which satisfies the following three closure conditions:
\begin{enumerate}
    \item $\square^0 \in \calC$;
    \item If $C \in (\calC \cap \sCat_{N-1,\ast,\ast})$, then $\Sigma C \in \calC$;
    \item If $C,D \in \calC$, then $C \vee D \in \calC$.
\end{enumerate}
\end{lem}
\begin{proof}
We have observed in \cref{def:susp} and \cref{eg:wedge-theta} that $\Theta_{N,\ast,\ast}^\nat$ satisfies (2) and (3) respectively; (1) is trivial. Conversely, if $\theta \in \Theta_{N,\ast,\ast}^\nat$, then $\theta$ decomposes as a wedge sum $\theta = \theta_1 \vee\cdots \vee \theta_k$ where each $\theta_i$ has exactly two objects. Moreover, if $\theta \in \Theta_{N,\ast,\ast}^\nat$ has exactly two objects, then $\theta = \Sigma \zeta$ is the suspension of some $\zeta \in \Theta_{N-1}$. Equipping $\zeta$ with its natural bipointing, the claim now follows for finite $N$ by induction on $N$, and thence also for $N = \omega$.
\end{proof}

\section{The main theorem}\label{sec:cubes}
Having discussed the structure of $\Theta$ in sufficient detail, our next aim is to remove $\Theta$ from the picture entirely, and reduce \cref{thm:mainthm} to the following two statements, which mention only cubes and are not about $\Theta$ at all:

\begin{thm}\label{thm:wedge}
For any $m,n \in \nats$, there are maps
\[
\begin{tikzcd}
\iota_{m,n} : \square^m \vee \square^n \ar[r,shift left] & \square^m \otimes^\lax \square^n : \rho_{m,n} \ar[l,shift left]
\end{tikzcd}
\]
exhibiting the wedge sum $\square^m \vee \square^n$ as a retract in $\sCat_{N,\ast,\ast}$ of $\square^{m+n}$.
\end{thm}

\begin{thm}\label{thm:susp}
For any $n \in \nats$, there are maps 
\[
\begin{tikzcd}
\phi_n : \Sigma \square^n \ar[r,shift left] & \square^{n+1} : \psi_n \ar[l, shift left]
\end{tikzcd}
\]
exhibiting the suspension $\Sigma \square^n$ as a retract in $\sCat_{N,\ast,\ast}$ of $\square^{n+1}$.
\end{thm}

\cref{thm:wedge} and \cref{thm:susp} will be proven in the next two sections. But first let us deduce our main theorems.

\begin{cor}\label{cor:idem-closed}
The idempotent completion $\widetilde{\square_{\ast,\ast}^\nat}$ of the naturally-bipointed cube category $\square_{\ast,\ast}^\nat$ is closed in $\sCat_{\ast,\ast}$ under suspensions and wedge sums.
\end{cor}
\begin{proof}
Let $X, Y \in \widetilde{\square_{\ast,\ast}^\nat}$. Then $X,Y$ are bipointed retracts of some naturally-bipointed cubes $\square^m, \square^n$. By \cref{thm:wedge}, $\square^m \vee \square^n$ is a bipointed retract of the naturally-bipointed $\square^{m+n}$. By functoriality of the wedge construction, $X \vee Y$ is also a bipointed retract of $\square^{m+n}$. Similarly, by functoriality of suspension, $\Sigma X$ is a bipointed retract of $\Sigma \square^{m}$, and hence a bipointed retract of the naturally-bipointed $\square^{m+1}$ by \cref{thm:susp}.
\end{proof}

\begin{cor}\label{thm:mainthm}
$\Theta_{\ast,\ast}^\nat \subseteq \widetilde{\square_{\ast,\ast}^\nat}$ and $\Theta \subseteq \widetilde{\square}$.
\end{cor}
\begin{proof}
As $\square^0 \in \square^\nat_{\ast,\ast}$, the first statement follows from \cref{lem:theta-ind} and \cref{cor:idem-closed}. The second statement follows by forgetting bipointings, since any bipointed retract is in particular a retract.
\end{proof}

\subsection{Applications to weak $(\infty,N)$-categories}\label{subsec:weak}

Although \cref{thm:mainthm} is a purely 1-categorical statement about strict $\omega$-categories, our motivating application is $\infty$-categorical in nature. First we need a lemma.

\begin{lem}\label{lem:theta-dense}
$\Theta$ is dense in $\Cat_\omega^\inc$ and in $\Cat_\omega$.
\end{lem}
\begin{proof}
As recalled in \cref{subsec:notation}, $\Cat_\omega^\inc$ and $\Cat_\omega$ are localizations of $\Psh(\Theta)$ (at the spine inclusions; respectively at the spine inclusions and Rezk maps). As such, the restricted nerve functors $\Cat_\omega^\inc \to \Psh(\Theta)$ and $\Cat_\omega \to \Psh(\Theta)$ are fully faithful as desired.
\end{proof}

\begin{cor}\label{cor:dense}
The cube category $\square$ and its idempotent completion $\widetilde \square$ are dense in $\Cat_\omega^\inc$ and in $\Cat_\omega$.
\end{cor}
\begin{proof}
We have full subcategory inclusions $\Theta \subset \square \subset \Cat_\omega \subset \Cat_\omega^\inc$. By \cref{lem:theta-dense}, it follows that $\widetilde\square \subset \Cat_\omega$ and $\widetilde \square \subset \Cat_\omega^\inc$ are dense inclusions. Because density is not affected by idempotent completion, $\square \subset \Cat_\omega$ and $\square \subset \Cat_\omega^\inc$ are dense inclusions as well.
\end{proof}

\begin{cor}\label{cor:dense2}
Let $N \in \nats \cup \{\omega\}$. The functors $\square \to \Cat_N^\inc$ and $\square \to \Cat_N$, which localize all $n$-morphisms for $n \geq N+1$, are dense.
\end{cor}
\begin{proof}
This follows from \cref{cor:dense} because $\Cat_N^\inc \subseteq \Cat_\omega^\inc$ and $\Cat_N \subseteq \Cat_\omega$ are full subcategory inclusions.
\end{proof}

\begin{warning}
When $N < \omega$, the dense functors of \cref{cor:dense2} are not fully faithful. It is not clear how to improve to a fully faithful functor (except perhaps by restricting back to $\Theta_N$). On the one hand, the nature of the retractions constructed in \cref{sec:wedge} is such that we do not have a bound on the dimension of the cubes necessary to contain all of the objects of $\Theta_N$ as retracts. And on the other hand, it is not clear whether the localization of $\square$ at the $n$-morphisms for $n \geq N+1$ results in a collection of strict $\omega$-categories.
\end{warning}

\begin{rmk}\label{rmk:defs}
We have defined $\Cat_\omega^\inc$ to be the localization of $\Psh(\Theta)$ at the spine inclusions. This agrees with the model-independent definition $\Cat_\omega^\inc = \varprojlim_N \Cat_N^\inc$. One way to see this is to observe that this limit of $(\infty,1)$-categories is equally the colimit $\varinjlim^{\Pr^L}_N \Cat_N^\inc$ in $\Pr^L$, where the colimit is along the inclusion functor (left adjoint to the forgetful functor). Moreover, if we write $S_{< N}$ for the set of spine inclusions of dimension $\leq N$, then by definition we have $\Cat_N^\inc = \Psh(\Theta_N) \cup_{\Psh(S_{< N+1} \times [1])}^{\Pr^L} \Psh(S_{< N+1})$, where the pushout is taken in $\Pr^L$. Therefore, 
\begin{align*}
    \varprojlim \Cat_N^\inc
    &= {\varinjlim}^{\Pr^L} \Cat_N^\inc \\
    &= {\varinjlim}^{\Pr^L} (\Psh(\Theta_N) \cup_{\Psh(S_{< N+1} \times [1])}^{\Pr^L} \Psh(S_{< N+1})) \\
    &= ({\varinjlim}^{\Pr^L}\Psh(\Theta_N)) \cup_{{\varinjlim}^{\Pr^L} \Psh(S_{< N+1} \times [1])}^{\Pr^L} {\varinjlim}^{\Pr^L} \Psh(S_{< N+1}) \\
    &= \Psh(\Theta) \cup^{\Pr^L}_{\Psh(S_{< \omega} \times [1])} \Psh(S_{< \omega}) \\
    &= \Cat_\omega^\inc
\end{align*}
The argument that $\Cat_\omega = \varprojlim \Cat_N$ is similar.
\end{rmk}

\section{Wedge sums}\label{sec:wedge}
The goal of this section is to prove \cref{thm:wedge}. The method is inductive in nature and unfolds over the course of this entire section, so let us give an overview; the following argument is recapitulated in the formal proof at the end of the section. We wish to construct section/retraction pairs
\[
\begin{tikzcd}
\iota_{m,n}: \square^m \vee \square^n \ar[r,shift left] & \square^m \otimes^\lax \square^n : \rho_{m,n} \ar[l, shift left]
\end{tikzcd}
\]

The section $\iota_{m,n}$ is straightforward to define (\cref{def:iota}): it is the map which puts $\square^m$ on the ``bottom" and $\square^n$ on the ``right" (where we visualize the first factor of the tensor product as the $x$-axis coordinate and the second factor as the $y$-axis coordinate), forming a backwards ``$L$" shape. To define the retraction $\rho_{m,n}$, we turn to induction:
Suppose that we have constructed $\rho_{m,n}$, and let us now construct $\rho_{m,n+1}$. Tensoring with $\square^1$, we have a retraction 
\[
\begin{tikzcd}
\iota_{m,n} \otimes^\lax \id_{\square^1} : (\square^m \otimes^\lax \square^1) \cup_{\square^1} (\square^n \otimes^\lax \square^1) \ar[r, shift left] & \square^m \otimes^\lax \square^n \otimes^\lax \square^1 : \rho_{m,n} \otimes^\lax \id_{\square^1} \ar[l, shift left]
\end{tikzcd}
\] (cf. \cref{rmk:reduce-wedge}). So we reduce to showing that there is a retraction pair
\[
\begin{tikzcd}
\eta_{m,n} : \square^m \vee \square^{n+1} = \square^m \cup_{\square^0} \square^{n+1} \ar[r,shift left] & (\square^m \otimes^\lax \square^1) \cup_{\square^1} (\square^n \otimes^\lax \square^1) : \zeta_{m,n} \ar[l, shift left]
\end{tikzcd}
\]

Again, the ``section" map $\eta_{m,n}$ is written down directly in 
\cref{lem:eta} 
(and shown to fit into a factorization $\iota_{m,n+1} = (\iota_{m,n} \otimes^\lax \id_{\square^1})\eta_{m,n}$).
The construction of the retract $\zeta_{m,n}$ (\ref{eqn:def-zeta}) uses the map $\rho_{m,1}$, which exists by induction.

We begin by defining the section maps $\iota_{m,n}$.

\begin{lem}\label{def:iota}
Let $m,n \in \nats$. There is a well-defined bipointed map
\begin{align}
    \iota_{m,n} &: \square^m \vee \square^n \to \square^m \otimes^\lax \square^n \nonumber \\
    \iota_{m,n} &= (\id_{\square^m} \otimes^\lax \const_{\vec 0}) \cup_{\vec 1 \otimes^\lax \vec 0} (\const_{\vec 1} \otimes^\lax \id_{\square^n})
\end{align}
\end{lem}

\begin{proof}
The claim is that the following diagram commutes:
\begin{equation*}
    \begin{tikzcd}[row sep = 30]
        \square^0 \ar[r,"\vec 0"] \ar[d,equal] &
        \square^m \ar[d,"\id_{\square^m} \otimes^\lax \const_{\vec 0}"{description}] & 
        \square^0 \ar[l,"\vec 1" above] \ar[d,"\vec 1 \otimes^\lax \vec 0" description] \ar[r,"\vec 0" above] &
        \square^n \ar[d,"\const_{\vec 1} \otimes^\lax \id_{\square^n}"{description}] &
        \square^0 \ar[l,"\vec 1"] \ar[d,equal]
        \\
        \square^0 \ar[r,"\vec 0"] &
        \square^m \otimes^\lax \square^n & 
        \square^m \otimes^\lax \square^n \ar[r,equal] \ar[l,equal] & 
        \square^m \otimes^\lax \square^n &
        \square^0 \ar[l,"\vec 1"]
    \end{tikzcd}
\end{equation*}
This is true by inspection.
\end{proof}

\begin{lem}\label{lem:base-wedge}
Suppose that $m=0$ or $n=0$ or $m=n=1$. Then $\iota_{m,n}$ admits a bipointed retraction $\rho_{m,n} : \square^m \otimes^\lax \square^n \to \square^m \vee \square^n$.
\end{lem}
\begin{proof}
When $m=0$ or $n=0$, $\iota_{m,n}$ is the identity, so trivially $\rho_{m,n}$ exists and is likewise the identity. When $m=n=1$, the claim is that $\iota_{1,1} : \square^1 \vee \square^1 \to \square^2$ admits a bipointed retract. This is clear: for instance, we may define $\rho_{1,1}$ to send the point not in $\iota_{1,1}(\square^1 \vee \square^1)$ to the join point. Then $\rho$ is clearly well-defined on 1-skeleta, and it simply carries the nonidentity 2-cell to the identity on a (composite) 1-cell.
\end{proof}

\begin{rmk}\label{rmk:reduce-wedge}
Assume (e.g. for induction as in the proof of \cref{thm:wedge} below -- once we have proven \cref{thm:wedge} we will know that this assumption is always true) that $\iota_{m,n} : \square^m \vee \square^n \to \square^m \otimes^\lax \square^n$ admits a bipointed retraction $\rho_{m,n}$.

Then (cf. \cref{rmk:tensor-bipoint}) we have a canonical bipointing for $(\square^m \otimes^\lax \square^1) \cup_{\square^1} (\square^n \otimes^\lax \square^1)$, and the pair
\begin{equation}
    \iota_{m,n} \otimes^\lax \id_{\square^1} : (\square^m \otimes^\lax \square^1) \cup_{\square^1} (\square^n \otimes^\lax \square^1) {}^\to_\leftarrow  \square^m \otimes^\lax \square^n \otimes^\lax \square^1 : \rho_{m,n} \otimes^\lax \id_{\square^1}
\end{equation}
is a bipointed retraction diagram.
\end{rmk}

\begin{lem}\label{lem:eta}
Let $m,n \in \nats$. There is a well-defined bipointed map
\begin{align}
    \eta_{m,n} &: \square^m \vee \square^{n+1} \to (\square^m \otimes^\lax \square^1) \cup_{\square^1} (\square^n \otimes^\lax \square^1) \nonumber \\
    \eta_{m,n} &= (\id_{\square^m} \otimes^\lax 0) \cup_{0} (\id_{\square^n \otimes^\lax \square^1})
\end{align}

Moreover, we have a factorization
\begin{equation}
    \iota_{m,n+1} = (\iota_{m,n} \otimes^\lax \id_{\square^1}) \eta_{m,n}
\end{equation}
\end{lem}
\begin{proof}
Consider the following maps of spans:
\begin{equation}\label{eqn:eta}
    \begin{tikzcd}
        \square^m \ar[d,"\id_{\square^m} \otimes^\lax 0" swap] &
        \square^0 \ar[l,"\vec 1" above] \ar[r,"\vec 0"] \ar[d,"0"] &
        \square^{n+1} \ar[d,equal] \\
        \square^m \otimes^\lax \square^1 \ar[d,"\id_{\square^m} \otimes^\lax \vec 0 \otimes^\lax \id_{\square^1}" swap] &
        \square^1 \ar[r,"\vec 0 \otimes^\lax \id_{\square^1}" above] \ar[l,"\vec 1 \otimes^\lax \id_{\square^1}" above] \ar[d,"\vec 1 \otimes^\lax \vec 0 \otimes^\lax \id_{\square^1}"] &
        \square^{n} \otimes^\lax \square^1 \ar[d,"\vec 1 \otimes^\lax \id_{\square^n} \otimes^\lax \id_{\square^1}" ] \\
        \square^m \otimes^\lax \square^{n} \otimes^\lax \square^1 & 
        \square^m \otimes^\lax \square^{n} \otimes^\lax \square^1 \ar[l,equal] \ar[r,equal] & 
        \square^m \otimes^\lax \square^{n} \otimes^\lax \square^1
    \end{tikzcd}
\end{equation}
The pushout of the upper map of spans is the map $\eta_{m,n}$, so it is well-defined (and bipointedness is easy to see). The pushout of the lower map of spans is $\iota_{m,n} \otimes^\lax \id_{\square^1}$, and the pushout of the composite map of spans is $\iota_{m,n+1}$. This gives the desired factorization.
\end{proof}

\begin{proof}[Proof of \cref{thm:wedge}]
We wish to show that $\iota_{m,l}$ (as defined in \cref{def:iota}) admits a bipointed retraction $\rho_{m,l}$. We proceed by induction on $m+l$. When $m+l \leq 2$, this follows from \cref{lem:base-wedge}. Suppose now that $m+l \geq 3$.
Without loss of generality we have $l\geq 2$ (the argument in the case where $l=1$ and $m\geq 2$ is similar to the following, but uses the dual of \cref{lem:eta}. Let $n=l-1$, so that we are trying to construct $\rho_{m,n+1}$. By induction, $\iota_{m,n}$ has a bipointed retract $\rho_{m,n}$, so that (cf. \cref{rmk:reduce-wedge}) $\iota_{m,n} \otimes^\lax \id_{\square^1}$ has a bipointed retract $\rho_{m,n} \otimes^\lax \id_{\square^1}$. By \cref{lem:eta}, we have a factorization $\iota_{m,n+1} = (\iota_{m,n} \otimes^\lax \id_{\square^1}) \eta_{m,n}$. Therefore, it will suffice to show that the map $\eta_{m,n}$ admits a bipointed retract $\zeta_{m,n}$ like so:
\begin{equation}
    \eta_{m,n} : \square^m \vee \square^{n+1} {}^\to_\leftarrow (\square^m \otimes^\lax \square^1) \cup_{\square^1} (\square^n \otimes^\lax \square^1) : \zeta_{m,n}
\end{equation}
To this end, we contemplate the following commutative diagram:
\begin{equation}\label{eqn:def-zeta}
\begin{tikzcd}
    \square^m \ar[d,"\id_{\square^m \otimes^\lax 0}" swap] 
    & 
    \square^0 \ar[l,"\vec 1" swap] \ar[d,"0"] \ar[r,"\vec 0"] &
    \square^{n+1} \ar[d,equal] 
    \\
    \square^{m}\otimes^\lax \square^1 \ar[d,"\rho_{m,1}" swap] & 
    \square^1 \ar[l, "\vec 1 \otimes^\lax \id_{\square^1}" swap] \ar[r,"\vec 0 \otimes^\lax \id_{\square^1}"] \ar[d] &
    \square^{n+1} \ar[dd] \\
    \square^m \vee \square^1 \ar[d, "\id_{\square^m} \vee (\vec 0 \otimes^\lax \id_{\square^1})" swap] & 
    \square^m \vee \square^1 \ar[l,equal] \ar[d, "\id_{\square^m} \vee (\vec 0 \otimes^\lax \id_{\square^1})"] \ar[ul,"\iota_{m,1}" swap] & \\
    \square^m \vee \square^{n+1} & \square^m \vee \square^{n+1} \ar[l,equal] \ar[r,equal] & \square^m \vee \square^{n+1}
\end{tikzcd}
\end{equation}
(Note that we have used the inductive hypothesis to invoke the existence of the retract $\rho_{m,1}$ of $\iota_{m,1}$.) From the first row to the second row we have a map of spans, and $\eta_{m,n}$ is the map which results upon passing to pushouts. From the second to the bottom row we have a map of spans; $\zeta_{m,n}$ is defined to be the map resulting by passing to pushouts. From the top row to the bottom row we again have a composite map of spans, and one checks via diagram chase that the induced map of pushouts is the identity, so that $\zeta_{m,n}$ is a retract of $\eta_{m,n}$ as desired. Bipointedness is straightforward to verify.
\end{proof}

\section{Suspensions}\label{sec:susp}
The goal of this section is to prove \cref{thm:susp}, i.e. to construct a bipointed retraction pair 
\begin{equation}
\begin{tikzcd}
    \phi_n : \Sigma \square^n \ar[r, shift left] & \square^n \otimes^\lax \square^1 : \psi_n \ar[l, shift left]
\end{tikzcd}
\end{equation}
The argument will be by induction on $n$, and similar in flavor to the proof of \cref{thm:wedge} in the previous section. Let us give an overview of the argument, which is recapitulated in the proof at the very end of the section. In the case $n=0$, the statement is obvious -- we take $\phi_0 = \psi_0 = \id_{\square^1}$. The case $n=1$ claims that the 2-globe $\Sigma \square^1$ is a bipointed retract of $\square^2$. This is also easy to verify by hand, but defining the ``section" map $\phi_1$ is already a bit fiddly, in that one needs to know something about composites of 1-morphisms in $\square^2$. By contrast, the definition of the ``retraction" map $\psi_1$ is more straightforward. Thus we begin (\cref{obs:2}) by defining $\psi_n : \square^{n+1} \to \Sigma \square^n$ for all $n$ using the pushout formula $\Sigma \square^n = \square^{n+1} \cup_{\partial \square^1 \otimes^\lax \square^n} \partial \square^1$ of \cref{lem:susp-po} -- $\psi_n$ is simply defined to be the canonical quotient map appearing in the pushout square. It remains to inductively construct sections $\phi_n$ of the maps $\psi_n$. Suppose for induction that we have constructed $\phi_{n}$, and let us now construct $\phi_{n+1}$. Tensoring with $\square^1$, we have a retraction 
\[
\begin{tikzcd}
\phi_{n} \otimes^\lax \id_{\square^1} : \Sigma \square^n \otimes^\lax \square^1 \ar[r, shift left] & \square^n \otimes^\lax \square^1 \otimes^\lax \square^1 : \psi_{n} \otimes^\lax \id_{\square^1} \ar[l, shift left]
\end{tikzcd}
\]
(cf. \cref{rmk:reduce-wedge}). So we reduce to showing that there is a retraction pair
\[
\begin{tikzcd}
\chi_n : \Sigma \square^{n+1} \ar[r,shift left] &  \Sigma \square^{n} \otimes^\lax \square^1 : \xi_{n} \ar[l, shift left]
\end{tikzcd}
\]
Using the pushout decomposition of \cref{lem:susp-po}, we see that (\cref{obs:1-loc}) we have an obvious choice for $\xi_n$ which localizes at two explicit 1-morphisms and nothing more. After analyzing this localization in \cref{eg:po-anal}, we see that there is a unique such section in \cref{lem:sec}.

\begin{Def}\label{obs:2}
Recall from \cref{lem:susp-po} that we have
$$\Sigma \square^n = \square^{n+1} \cup_{\partial \square^1 \otimes^\lax \square^n} \partial \square^1$$
For each $n$, we take the ``retraction" map $\psi_n: \square^{n+1} \to \Sigma \square^n$ in the statement of \cref{thm:susp} to be the canonical quotient map appearing in the above pushout square.
\end{Def}

\begin{obs}\label{obs:1-loc}
Note that we have
\begin{align*}
    \Sigma \square^{n+1} 
    &= \square^{n+2} \cup_{\partial \square^{1} \otimes^\lax \square^{n+1}} \partial \square^1 \\
    &= \square^{n+2} \cup_{\partial \square^{1} \otimes^\lax \square^{n+1}} (\partial \square^1 \otimes^\lax \square^1) \cup_{\partial \square^1 \otimes^\lax \square^1} \partial \square^1 \\
    &= ((\Sigma \square^n) \otimes^\lax \square^1)\cup_{\partial \square^1 \otimes^\lax \square^1} \partial \square^1
\end{align*}

That is, we have a canonical map $\xi_n : (\Sigma \square^n) \otimes^\lax \square^1 \to \Sigma \square^{n+1}$ which is given by collapsing two 1-morphisms to points. Moreover, we have a bipointed factorization $\psi_{n+1} = \xi_n(\psi_n \otimes^\lax \id_{\square^1})$.
\end{obs}

For the next two lemmas, we fix $n \in \nats$ and assume the following inductive hypothesis:
\begin{equation}\label{eqn:ind}
    \text{$\Sigma \square^n$ is a retract of a cube.}
\end{equation}
(Once we have proven \cref{thm:susp}, we will know that Hypothesis \ref{eqn:ind} is always true.) The upshot of Hypothesis \ref{eqn:ind} is that by \cref{lem:stupid-lem}, it follows that for each object $x \in \{0,1\} = \Ob(\Sigma \square^n)$, we have
\begin{equation}\label{eqn:contr}
\Hom_{\Sigma \square^n \otimes^\lax \square^1}((x,0),(x,0)) = \Hom_{\Sigma \square^n \otimes^\lax \square^1}((x,0),(x,1)) = \Hom_{\Sigma \square^n \otimes^\lax \square^1}((x,1),(x,1)) = \square^0
\end{equation}

\begin{lem}\label{eg:po-anal}
(Assuming Hypothesis \ref{eqn:ind}.)  The action of $\xi_n$ on the ``long" hom-category is an isomorphism
\[\Hom_{(\Sigma \square^n) \otimes^\lax \square^1}(00,11) \xrightarrow \cong \Hom_{\Sigma (\square^{n+1})}(0,1)\]
\end{lem}
\begin{proof}
Write $C = \Sigma (\square^{n+1})$ and $B =  (\Sigma \square^n) \otimes^\lax \square^1$. We have 
\begin{equation}\label{eqn:meh}
C = B \cup_{\partial \square^1 \otimes^\lax \square^1} \partial \square^1
\end{equation}
by \cref{obs:1-loc}. The $(n+2)$-category $B$ has 4 objects, which we label $00,01,10,11$. There are canonical 1-morphisms $f : 00 \to 01, g : 10 \to 11$ which we collapse to obtain the two-object $(n+2)$-category $C$. We denote the two objects of $C$ by $0,1$. From the pushout description \ref{eqn:meh}, we have that
\[\Hom_C(0,1) = \Hom_B(00,11) \amalg (\Hom_B(00,01) \times \Hom_B(10,11)) \amalg (\Hom_B(00,10) \times \Hom_B(10,11)) / \sim\]
where ``$\sim$" is the congruence identifying $(f,\phi)$ with $f\phi$ and identifying $(\psi,g)$ with $\psi g$. By Hypothesis \ref{eqn:ind} and the resulting \ref{eqn:contr}, we have that $\Hom_B(00,01) = \{f\}$ and $\Hom_B(10,11) = \{g\}$ are each terminal. It follows that there is a normal form for elements of equivalence classes under the congruence ``$\sim$", which is simply to take the representative in $\Hom_B(00,11)$. That is, we have $\Hom_C(0,1) = \Hom_B(00,11)$ as desired.
\end{proof}

\begin{lem}\label{lem:sec}
(Assuming Hypothesis \ref{eqn:ind}.) the map $\xi_n : (\Sigma \square^n) \otimes^\lax \square^1 \to \Sigma \square^{n+1}$ of \cref{obs:1-loc} has a bipointed section $\chi_n : \Sigma \square^{n+1} \to (\Sigma \square^n) \otimes^\lax \square^1$.
\end{lem}
\begin{proof}
By \cref{lem:susp-univ}, to define $\chi_n$ it suffices to choose a map $\square^{n+1} \to \Hom_{(\Sigma \square^n) \otimes^\lax \square^1}(00,11)$. We take this to be the inverse of the isomorphism from \cref{eg:po-anal}. So by \cref{eg:po-anal} it is indeed a section to $\xi_n$ (in fact, the unique section).
\end{proof}

\begin{proof}[Proof of \cref{thm:susp}]
We wish to construct bipointed sections $\phi_n : \Sigma (\square^n) \to \square^1 \otimes^\lax \square^n$ of the maps $\psi_n$ of \cref{obs:2}, by induction on $n$. When $n=0$, we take $\phi_0$ to be the identity. Assume for induction that we have defined $\phi_n$; let us define $\phi_{n+1}$. Then $\phi_n \otimes^\lax \id_{\square^1}$ is a bipointed section of $\psi_n \otimes^\lax \id_{\square^1}$. By \cref{obs:1-loc} we have $\psi_{n+1} = \xi_n (\psi_n \otimes^\lax \id_{\square^1})$, so it will suffice to construct a bipointed section $\chi_n$ of $\xi_n$. This is \cref{lem:sec} (which was proven under Hypothesis \ref{eqn:ind}, which is part of our inductive hypothesis now).
\end{proof}

\appendix

\section{A small computation with cubes}\label{sec:app}
In \cref{eg:po-anal}, we have used the following:
\begin{lem}\label{lem:stupid-lem}
Let $X \in \widetilde{\square}$ be a retract of a cube, and let $x \in X$. Then the following hom-categories are all terminal:
\[\Hom_{X \otimes^\lax \square^1}((x,0),(x,0)) = \Hom_{X \otimes^\lax \square^1}((x,0),(x,1)) = \Hom_{X \otimes^\lax \square^1}((x,1),(x,1)) = \square^0\]
\end{lem}
This is surely well-known. We will deduce it from the following stronger statement:

\begin{lem}\label{lem:next-stupid}
Let $X \in \widetilde{\square}$ be a retract of a cube, and let $x \in X$. Consider the inclusion 
\[
\{x\} \otimes^\lax \id_{\square^n} \otimes^\lax \{y\}: \square^n \to \square^m \otimes^\lax \square^n \otimes^\lax \square^p
\]
This inclusion functor is fully faithful for any $x \in \square^n, y \in \square^p$.
\end{lem}

\begin{proof}[Proof of \cref{lem:stupid-lem}]
It suffices to treat the case where $X = \square^m$ is a cube. In other words, it suffices to show that the hom-categories $\Hom_{\square^{m} \otimes^\lax \square^{1}}((x,0),(x,0)) = \Hom_{\square^{m} \otimes^\lax \square^{1}}((x,0),(x,1)) = \Hom_{\square^{m} \otimes^\lax \square^{1}}((x,1),(x,1)) = \square^0$ are all terminal. By \cref{lem:next-stupid}, it will suffice to treat the case $m = 0$. Thus we are reduced to showing that $\Hom_{\square^1}(0,0) = \Hom_{\square^1}(0,1) = \Hom_{\square^1}(1,1) = \square^0$, which is clear because $\square^1$ is the arrow category.
\end{proof}

\cref{lem:next-stupid} follows from the following statement and its dual:
\begin{lem}\label{lem:next-next-stupid}
Let $X \in \widetilde{\square}$ be a retract of a cube, and let $x \in X$. The inclusion 
\[
\{x\} \otimes^\lax \id_{\square^n} : \square^n \to \square^m \otimes^\lax \square^n
\]
is fully faithful for any $x \in \square^m$.
\end{lem}

\begin{proof}[Proof of \cref{lem:next-stupid}]
We have a factorization $\{x\} \otimes^\lax \id_{\square^n} \otimes^\lax \{y\} = (\{x\} \otimes^\lax \id_{\square^{n+p}})(\id_{\square^n} \otimes^\lax \{y\})$. By \cref{lem:next-next-stupid}, and its dual, both factors are fully faithful.
\end{proof}

\begin{proof}[Proof of \cref{lem:next-next-stupid}]
By factoring $\square^n \to \square^m \otimes^\lax \square^n$ as $\square^n \to \square^1 \otimes^\lax \square^n \to \cdots \to \square^m \otimes^\lax \square^n$, we reduce to the case $m = 1$.

The map $\partial^0 \otimes^\lax \id_{\square^n} : \square^n \to \square^1 \otimes^\lax \square^n$ is split monic and hence injective on hom-categories. Moreover, $\square^1$ and $\square^n$ are both Steiner complexes (\cite{steiner}) with bases $B_1,B_n$. Here $B_1 = \{\partial^0,\partial^1,\iota\}$, where $\partial^0,\partial^1$ have degree $0$ and $\iota$ has degree $1$. It follows (cf \cite[2.4]{ara-lucas}) that $\square^1 \otimes^\lax \square^n$ is a Steiner complex with basis $\{x \otimes y \mid x \in B_1,\, y \in B_n\}$. All morphisms in $\square^1 \otimes \square^n$ are built from this basis, and the image of $\square^n$ is generated by $\{\partial^0 \otimes y \mid y \in B_n\}$. Any other morphism must therefore have a codomain whose first coordinate is $1$ rather than $0$, and so does not have codomain in the image of $\partial^0 \otimes^\lax \id_{\square^n}$. Thus $\partial^0 \otimes^\lax \id_{\square^n}$ is also surjective on hom-categories, and so is fully faithful as desired. The argument for $\partial^1 \otimes^\lax \id_{\square^n}$ is similar.
\end{proof}


\end{document}